\documentclass[10pt]{article}

\usepackage{mathptmx,eucal,amsmath,amscd,amssymb,amsthm,xspace}
\usepackage[all,tips]{xy}
\usepackage[dvips]{graphicx}
\usepackage{graphics,epsfig,wrapfig,verbatim,syntonly}
\usepackage{hyperref,amssymb,color, url,fancyhdr}

\theoremstyle{plain}

\newtheorem{thm}{Theorem}[section]

\newtheorem{cor}[thm]{Corollary}

\newtheorem{prop}[thm]{Proposition}

\newtheorem{lemma}[thm]{Lemma}
\newtheorem{sch}[thm]{Scholium}

\theoremstyle{definition}

\DeclareMathOperator{\GL}{GL}

\DeclareMathOperator{\LCM}{LCM}\DeclareMathOperator{\lcm}{lcm}
\DeclareMathOperator{\D}{D}\DeclareMathOperator{\G}{G}

\DeclareMathOperator{\DM}{F}\DeclareMathOperator{\word}{w}
\DeclareMathOperator{\sub}{s}

\newcommand{\vp}{\varphi}

\newcommand{\nid}{\noindent}

\newcommand{\iny}{\infty}

\newcommand{\es}{\emptyset}
\newcommand{\co}{\ensuremath{\colon}}

\newcommand{\innp}[1]{\left< #1 \right>}
\newcommand{\abs}[1]{\left\vert#1\right\vert}
\newcommand{\set}[1]{\left\{#1\right\}}
\newcommand{\norm}[1]{\left\vert \left\vert #1\right\vert\right\vert}

\newcommand{\pr}[1]{\left( #1 \right) }
\newcommand{\su}{\subset}

\newcommand{\ba}{\bigcap}

\newcommand{\lra}{\longrightarrow}

\newcommand{\B}[1]{\ensuremath{\mathbf{#1}}}

\newcommand{\N}{\ensuremath{\B{N}}}

\newcommand{\R}{\ensuremath{\B{R}}}
\newcommand{\Z}{\ensuremath{\B{Z}}}

\fancypagestyle{plain}

\begin{document}
\bibliographystyle{plain}


\title{\textbf{Asymptotic growth and \\least common multiples in groups}}
\author{K. Bou-Rabee\thanks{University of Chicago, Chicago, IL 60637. E-mail: \tt{khalid@math.uchicago.edu}}~ and D. B. McReynolds\thanks{University of Chicago, Chicago, IL 60637. E-mail: \tt{dmcreyn@math.uchicago.edu}}}
\maketitle


\begin{abstract}
\nid In this article we relate word and subgroup growth to certain functions that arise in the quantification of residual finiteness. One consequence of this endeavor is a pair of results that equate the nilpotency of a finitely generated group with the asymptotic behavior of these functions. The second half of this article investigates the asymptotic behavior of two of these functions. Our main result in this arena resolves a question of Bogopolski from the Kourovka notebook concerning lower bounds of one of these functions for nonabelian free groups.
\end{abstract}

\smallskip\smallskip

\nid 1991 MSC classes: 20F32, 20E26
\vskip.05in
\nid keywords: \emph{free groups, hyperbolic groups, residual finiteness, subgroup growth, word growth.}

\section{Introduction}

\nid The goals of the present article are to examine the interplay between word and subgroup growth, and to quantify residual finiteness, a topic motivated and described by the first author in \cite{Bou}. These two goals have an intimate relationship that will be illustrated throughout this article. \smallskip\smallskip

\nid Our focus begins with the interplay between word and subgroup growth. Recall that for a fixed finite generating set $X$ of $\Gamma$ with associated word metric $\norm{\cdot}_X$, word growth investigates the asymptotic behavior of the function
\[ \word_{\Gamma,X}(n) = \abs{\set{\gamma \in \Gamma~:~ \norm{\gamma}_X \leq n}}, \]
while subgroup growth investigates the asymptotic behavior of the function
\[ \sub_\Gamma(n) = \abs{\set{\Delta \lhd \Gamma~:~ [\Gamma:\Delta]\leq n}}. \]
To study the interaction between word and subgroup growth we propose the first of a pair of questions: \smallskip\smallskip

\nid \textbf{Question 1.} \emph{What is the smallest integer $\DM_{\Gamma,X}(n)$ such that for every word $\gamma$ in $\Gamma$ of word length at most $n$, there exists a finite index normal subgroup of index at most $\DM_{\Gamma,X}(n)$ that fails to contain $\gamma$?} \smallskip\smallskip

\nid To see that the asymptotic behavior of $\DM_{\Gamma,X}(n)$ measures the interplay between word and subgroup growth, we note the following inequality (see Section \ref{Preliminary} for a simple proof):
\begin{equation}\label{BasicInequality}
\log (\word_{\Gamma,X}(n)) \leq \sub_\Gamma(\DM_{\Gamma,X}(2n))\log (\DM_{\Gamma,X}(2n)).
\end{equation}
Our first result, which relies on Inequality (\ref{BasicInequality}), is the following.

\begin{thm}\label{DivisibilityLogGrowth}
 If $\Gamma$ is a finitely generated linear group, then the following are equivalent:
\begin{itemize}
\item[(a)]
$\DM_{\Gamma,X}(n) \leq (\log(n))^r$ for some $r$.
\item[(b)]
$\Gamma$ is virtually nilpotent.
\end{itemize}
\end{thm}

\nid For finitely generated linear groups that is not virtually nilpotent, Theorem \ref{DivisibilityLogGrowth} implies $\DM_{\Gamma,X}(n) \nleq (\log(n))^r$ for any $r >0$. For this class of groups, we can improve this lower bound. Precisely, we have the following result---see Section \ref{Preliminary} for the definition of $\preceq$.

\begin{thm}\label{basiclowerbound}
Let $\Gamma$ be a group that contains a nonabelian free group of rank $m$. Then
\[ n^{1/3} \preceq \DM_{\Gamma,X}(n).\]
\end{thm}

\nid The motivation for the proof of Theorem \ref{basiclowerbound} comes from the study of $\DM_{\Z,X}(n)$, where the Prime Number Theorem and least common multiples provide lower and upper bounds for $\DM_{\Z,X}(n)$. In Section \ref{FreeGroupGrowth}, we extend this approach by generalizing least common multiples to finitely generated groups (a similar approach was also taken in the article of Hadad \cite{Hadad}). Indeed with this analogy, Theorem \ref{basiclowerbound} and the upper bound of $n^3$ established in \cite{Bou}, \cite{Rivin} can be viewed as a weak Prime Number Theorem for free groups since the Prime Number Theorem yields $\DM_\Z(n) \simeq \log(n)$. Recently, Kassabov--Matucci \cite{KM} improved the lower bound of $n^{1/3}$ to $n^{2/3}$. A reasonable guess is that $\DM_{F_m,X}(n) \simeq n$, though presently neither the upper or lower bound is known. We refer the reader to \cite{KM} for additional questions and conjectures.\smallskip\smallskip

\nid There are other natural ways to measure the interplay between word and subgroup growth. Let $B_{\Gamma,X}(n)$ denote $n$--ball in $\Gamma$ for the word metric associated to the generating set $X$. Our second measurement is motivated by the following question---in the statement, $B_{\Gamma,X}(n)$ is the metric $n$--ball with respect to the word metric $\norm{\cdot}_X$:\smallskip\smallskip

\nid \textbf{Question 2.} \emph{What is the cardinality $\G_{\Gamma,X}(n)$ of the smallest finite group $Q$ such that there exists a surjective homomorphism $\vp\co \Gamma \to Q$ with the property that $\vp$ restricted to $B_{\Gamma,X}(n)$ is injective?}\smallskip\smallskip

\nid We call $\G_{\Gamma,X}(n)$ the \emph{residual girth function} and relate $\G_{\Gamma,X}(n)$ to $\DM_{\Gamma,X}$ and $\word_{\Gamma,X}(n)$ for a class of groups containing non-elementary hyperbolic groups; Hadad \cite{Hadad} studied group laws on finite groups of Lie type, a problem that is related to residual girth and the girth of a Cayley graph for a finite group. Specifically, we obtain the following inequality (see Section \ref{FreeGroupGrowth} for a precise description of the class of groups for which this inequality holds):
\begin{equation}\label{BasicGirthEquation}
\G_{\Gamma,X}(n/2) \leq  \DM_{\Gamma,X}\pr{6n(\word_{\Gamma,X}(n))^{2}}.
\end{equation}
Our next result shows that residual girth functions enjoy the same growth dichotomy as word and subgroup growth---see \cite{gromov} and \cite{lubsegal-2003}.

\begin{thm}\label{GirthPolynomialGrowth}
If $\Gamma$ is a finitely generated group then the following are equivalent.
\begin{itemize}
\item[(a)]
$\G_{\Gamma,X}(n) \leq n^r$ for some $r$.
\item[(b)]
$\Gamma$ is virtually nilpotent.
\end{itemize}
\end{thm}

\nid The asymptotic growth of $\DM_{\Gamma,X}(n)$, $\G_{\Gamma,X}(n)$, and related functions arise in quantifying residual finiteness, a topic introduced in \cite{Bou} (see also the recent articles of the authors \cite{BM}, Hadad \cite{Hadad}, Kassabov--Mattucci \cite{KM}, and Rivin \cite{Rivin}). Quantifying residual finiteness amounts to the study of so-called divisibility functions. Given a finitely generated, residually finite group $\Gamma$, we define the \emph{divisibility function} $\D_\Gamma\co \Gamma^\bullet \lra \N $ by
\[ \D_\Gamma(\gamma) = \min \set{[\Gamma:\Delta] ~:~ \gamma \notin \Delta}. \]
The associated \emph{normal divisibility function} for normal, finite index subgroups is defined in an identical way and will be denoted by $\D_{\Gamma}^\lhd$. It is a simple matter to see that $\DM_{\Gamma,X}(n)$ is the maximum value of $\D_{\Gamma}^\lhd$ over all non-trivial elements in $B_{\Gamma,X}(n)$. We will denote the associated maximum of $\D_\Gamma$ over this set by $\max \D_\Gamma (n)$.\smallskip\smallskip

\nid The rest of the introduction is devoted to a question of Oleg Bogopolski, which concerns $\max \D_{\Gamma,X}(n)$. It was established in \cite{Bou} that $\log(n) \preceq \max \D_{\Gamma,X}(n)$ for any finitely generated group with an element of infinite order (this was also shown by \cite{Rivin}). For a nonabelian free group $F_m$ of rank $m$, Bogopolski asked whether $\max \D_{F_m,X}(n) \simeq \log(n)$ (see Problem 15.35 in the Kourovka notebook \cite{TheBook}). Our next result answers Bogopolski's question in the negative---we again refer the reader to Section \ref{Preliminary} for the definition of $\preceq$.

\begin{thm}\label{toughlowerbound}
If $m>1$, then $\max \D_{F_m,X}(n) \npreceq \log(n)$.
\end{thm}

\nid We prove Theorem \ref{toughlowerbound} in Section \ref{toughlowerboundSection} using results from Section \ref{FreeGroupGrowth}. The first part of the proof of Theorem \ref{toughlowerbound} utilizes the material established for the derivation of Theorem \ref{basiclowerbound}. The second part of the proof of Theorem \ref{toughlowerbound} is topological in nature, and involves a careful study of finite covers of the figure eight. It is also worth noting that our proof only barely exceeds the proposed upper bound of $\log(n)$. In particular, at present we cannot rule out the upper bound $(\log(n))^2$. In addition, to our knowledge the current best upper bound is $n/2 + 2$, a result established recently by Buskin \cite{Bus}. In comparison to our other results, Theorem \ref{toughlowerbound} is the most difficult to prove and is also the most surprising. Consequently, the reader should view Theorem \ref{toughlowerbound} as our main result.\smallskip\smallskip

\paragraph{\textbf{Acknowledgements.}}

Foremost, we are extremely grateful to Benson Farb for his inspiration, comments, and guidance. We would like to thank Oleg Bogopolski, Emmanuel Breuillard, Jason Deblois, Jordan Ellenberg, Tsachik Gelander, Uzy Hadad, Fr\'{e}d\'{e}ric Haglund, Ilya Kapovich, Martin Kassabov, Larsen Louder, Justin Malestein, Francesco Matucci, and Igor Rivin for several useful conversations and their interest in this article. Finally, we extend thanks to Tom Church, Blair Davey, and Alex Wright for reading over earlier drafts of this paper. The second author was partially supported by an NSF postdoctoral fellowship.

\section{Divisibility and girth functions}\label{Preliminary}

\nid In this introductory section, we lay out some of the basic results we require in the sequel. For some of this material, we refer the reader to \cite[Section 1]{Bou}.

\paragraph{\textbf{Notation.}} Throughout, $\Gamma$ will denote a finitely generated group, $X$ a fixed finite generating set for $\Gamma$, and $\norm{\cdot}_X$ will denote the word metric. For $\gamma \in \Gamma$, $\innp{\gamma}$ will denote the cyclic subgroup generated by $\gamma$ and $\overline{\innp{\gamma}}$ the normal closure of $\innp{\gamma}$.
For any subset $S \subset \Gamma$ we set $S^\bullet = S-1$.

\paragraph{\textbf{1. Function comparison and basic facts}.} For a pair of functions $f_1,f_2\co \N \to \N$, by $f_1 \preceq f_2$, we mean that there exists a constant $C$ such that $f_1(n) \leq Cf_2(Cn)$ for all $n$. In the event that $f_1 \preceq f_2$ and $f_2 \preceq f_1$, we will write $f_1 \simeq f_2$.\smallskip\smallskip

\nid This notion of comparison is well suited to the functions studied in this paper. We summarize some of the basic results from \cite{Bou} for completeness.

\begin{lemma}\label{DivisibilityAsymptoticLemma}
Let $\Gamma$ be a finitely generated group.
\begin{itemize}
\item[(a)]
If $X,Y$ are finite generating sets for $\Gamma$ then $\DM_{\Gamma,X} \simeq \DM_{\Gamma,Y}$.
\item[(b)]
If $\Delta$ is a finitely generated subgroup of $\Gamma$ and $X,Y$ are finite generating sets for $\Gamma,\Delta$ respectively, then $\DM_{\Delta,Y} \preceq \DM_{\Gamma,X}$.
\item[(c)]
If $\Delta$ is a finite index subgroup of $\Gamma$ with $X,Y$ as in (b), then $\DM_{\Gamma,X} \preceq (\DM_{\Delta,Y})^{[\Gamma:\Delta]}$.
\end{itemize}
\end{lemma}

\nid We also have a version of Lemma \ref{DivisibilityAsymptoticLemma} for residual girth functions.

\begin{lemma}\label{GirthAsymptoticLemma}
Let $\Gamma$ be a finitely generated group.
\begin{itemize}
\item[(a)]
If $X,Y$ are finite generating sets for $\Gamma$, then $\G_{\Gamma,X} \simeq \G_{\Gamma,Y}$.
\item[(b)]
If $\Delta$ is a finitely generated subgroup of $\Gamma$ and $X,Y$ are finite generating sets for $\Gamma,\Delta$ respectively, then $\G_{\Delta,Y} \preceq \G_{\Gamma,X}$.
\item[(c)]
If $\Delta$ is a finite index subgroup of $\Gamma$ with $X,Y$ as in (b), then $\G_{\Gamma,X} \preceq (\G_{\Delta,Y})^{[\Gamma:\Delta]}$.
\end{itemize}
\end{lemma}

\nid As the proof of Lemma \ref{GirthAsymptoticLemma} is straightforward, we have opted to omit it for sake of brevity. As a consequence of Lemmas \ref{DivisibilityAsymptoticLemma} and \ref{GirthAsymptoticLemma}, we occasionally suppress the dependence of the generating set in our notation.
\paragraph{\textbf{2. The basic inequality.}}

\nid We now derive (\ref{BasicInequality}) from the introduction. For the reader's convenience, recall (\ref{BasicInequality}) is
\[ \log (\word_{\Gamma,X}(n)) \leq \sub_\Gamma(\DM_{\Gamma,X}(2n))\log (\DM_{\Gamma,X}(2n)).\]

\begin{proof}[Proof of (\ref{BasicInequality})]
We may assume that $\Gamma$ is residually finite as otherwise $\DM_\Gamma(n)$ is eventually infinite for sufficiently large $n$ and the inequality is trivial. By definition, for each word $\gamma \in B_{\Gamma,X}^\bullet(2n)$, there exists a finite index, normal subgroup $\Delta_\gamma$ in $\Gamma$ such that $\gamma \notin \Delta_\gamma$ and $[\Gamma:\Delta_\gamma] \leq \DM_{\Gamma,X}(2n)$. Setting $\Omega_{\DM_{\Gamma,X}(2n)}(\Gamma)$ to be the intersection of all finite index, normal subgroup of index at most $\DM_{\Gamma,X}(2n)$, we assert that $B_{\Gamma,X}(n)$ injects into quotient $\Gamma/\Omega_{\DM_{\Gamma,X}(2n)}(\Gamma)$. Indeed, if two elements $\gamma_1,\gamma_2 \in B_{\Gamma,X}(n)$ had the same image, the element $\gamma_1\gamma_2^{-1}$ would reside in $\Omega_{\DM_{\Gamma,X}(2n)}(\Gamma)$. However, by construction, every element of word length at most $2n$ has nontrivial image. In particular, we see that
\begin{align*}
\word_{\Gamma,X}(n) &= \abs{B_{\Gamma,X}(n)} \leq \abs{\Gamma/\Omega_{\DM_{\Gamma,X}(2n)}(\Gamma)} \\
& \leq \prod_{\scriptsize{\begin{matrix} \Delta \lhd \Gamma \\ [\Gamma:\Delta]\leq \DM_{\Gamma,X}(2n)\end{matrix}}} \abs{\Gamma/\Delta} \\
&\leq \prod_{\scriptsize{\begin{matrix} \Delta \lhd \Gamma \\ [\Gamma:\Delta]\leq \DM_{\Gamma,X}(2n)\end{matrix}}} \DM_{\Gamma,X}(2n) \\
&\leq (\DM_{\Gamma,X}(2n))^{\sub_\Gamma(\DM_{\Gamma,X}(2n))}.
\end{align*}
Taking the log of both sides, we obtain
\[ \log(\word_{\Gamma,X}(n)) \leq \sub_\Gamma(\DM_{\Gamma,X}(2n))\log(\DM_{\Gamma,X}(2n)). \]
\end{proof}

\nid In fact, the proof of (\ref{BasicInequality}) yields the following.

\begin{sch}
Let $\Gamma$ be a finitely generated, residually finite group. Then
\[ \log (\G_{\Gamma,X}(n)) \leq \sub_\Gamma(\DM_{\Gamma,X}(2n)) \log(\DM_{\Gamma,X}(2n)). \]
\end{sch}

\paragraph{\textbf{3. An application of (\ref{BasicInequality}).}}

\nid We now derive the following as an application of (\ref{BasicInequality}).

\begin{prop}\label{BasicInequalityMainProp}
Let $\Gamma$ be a finitely generated, residually finite group. If there exists $\alpha > 1$ such that $\alpha^n \preceq \word_{\Gamma,X}(n)$, then $\DM_{\Gamma,X}(n) \npreceq (\log n)^r$ for any $r \in \R$.
\end{prop}

\begin{proof}
Assume on the contrary that there exists $r \in \R$ such that $\DM_{\Gamma,X} \preceq (\log(n))^r$. In terms of $\preceq$ notation, inequality (\ref{BasicInequality}) becomes:
\[ \log(\word_{\Gamma, X} (n)) \preceq s_\Gamma (\DM_{\Gamma, X}(n)) \log(\DM_{\Gamma, X}(n)). \]
Taking the log of both sides, we obtain
\[ \log\log(\word_{\Gamma, X} (n)) \preceq \log(\sub_\Gamma (\DM_{\Gamma, X}(n)))+  \log(\log(\DM_{\Gamma, X}(n))). \]
This inequality, in tandem with the assumptions
\begin{align*}
\alpha^n &\preceq \word_{\Gamma,X}(n), \\
\DM_{\Gamma,X}(n) &\preceq (\log(n))^r,
\end{align*}
and $\log(\sub_\Gamma(n)) \preceq (\log(n))^2$ (see \cite[Corollary 2.8]{lubsegal-2003}) gives
\[ \log(n) \preceq (\log\log(n))^2 + \log\log\log(n), \]
which is impossible.
\end{proof}

\nid With Proposition \ref{BasicInequalityMainProp}, we can now prove Theorem \ref{DivisibilityLogGrowth}.

\begin{proof}[Proof of Theorem \ref{DivisibilityLogGrowth}]
For the direct implication, we assume that $\Gamma$ is a finitely generated linear group with $\DM_\Gamma \preceq (\log n)^r$ for some $r$. According to the Tits' alternative, either $\Gamma$ is virtually solvable or $\Gamma$ contains a nonabelian free subgroup. In the latter case, $\Gamma$ visibly has exponential word growth and thus we derive a contradiction via Proposition \ref{BasicInequalityMainProp}. In the case $\Gamma$ is virtually solvable, $\Gamma$ must also have exponential word growth unless $\Gamma$ is virtually nilpotent (see \cite[Theorem VII.27]{harpe-2000}). This in tandem with Proposition \ref{BasicInequalityMainProp} implies $\Gamma$ is virtually nilpotent.\smallskip\smallskip

\nid For the reverse implication, let $\Gamma$ be a finitely generated, virtually nilpotent group with finite index, nilpotent subgroup $\Gamma_0$. According to Theorem 0.2 in \cite{Bou}, $\DM_{\Gamma_0} \preceq (\log n)^r$ for some $r$. Combining this with Lemma \ref{DivisibilityAsymptoticLemma} (c) yields $\DM_\Gamma \preceq (\log n)^{r[\Gamma:\Gamma_0]}$.
\end{proof}

\nid In the next two sections, we will prove Theorem \ref{basiclowerbound}. In particular, for finitely generated linear groups that are not virtually solvable, we obtain an even better lower bound for $\DM_{\Gamma,X}(n)$ than can be obtained using (\ref{BasicInequality}). Namely, $n^{1/3} \preceq \DM_{\Gamma,X}(n)$ for such groups. The class of non-nilpotent, virtually solvable groups splits into two classes depending on whether the rank of the group is finite or not. This is not the standard notion of rank but instead
\[ \textrm{rk}(\Gamma) = \max\set{ r(\Delta)~:~ \Delta \text{ is a finitely generated subgroup of } \Gamma}, \]
where
\[ r(\Delta) = \min\set{\abs{Y}~:~Y \text{ is a generating set for }\Delta}. \]
The class of virtually solvable groups with finite rank is known to have polynomial subgroup growth (see \cite[Chapter 5]{lubsegal-2003}) and thus have a polynomial upper bound on normal subgroup growth. Using this upper bound with (\ref{BasicInequality}) yields our next result.

\begin{cor}
If $\Gamma$ is virtually solvable, finite rank, and not nilpotent, then $n^{1/d} \preceq \DM_{\Gamma,X}(n)$ for some $d \in \N$.
\end{cor}

\begin{proof}
For a non-nilpotent, virtually solvable group of finite rank, we have the inequalities:
\begin{align*}
\alpha^n &\preceq \word_{\Gamma,X}(n) \\
\sub_{\Gamma,X}(n) &\preceq n^m.
\end{align*}
Setting $d=2m$ and assuming $\DM_{\Gamma,X}(n) \preceq n^{1/d}$, inequality (\ref{BasicInequality}) yields the impossible inequality
\[ n \simeq \log(\alpha^n) \preceq \log(\word_{\Gamma,X}(n)) \preceq \sub_\Gamma(\DM_{\Gamma,X}(n))\log(\DM_{\Gamma,X}(n)) \preceq (n^{1/d})^m \log(n^{1/d}) \simeq \sqrt{n}\log(n). \]
\end{proof}

\nid Virtually solvable group $\Gamma$ with infinite $\textrm{rk}(\Gamma)$ cannot be handled in this way as there exist examples with $c^{n^{1/d}} \preceq \sub_{\Gamma,X}(n)$ with $c>1$ and $d \in \N$. 


\section{Least common multiples}

\nid
Let $\Gamma$ be a finitely generated group and $S \su \Gamma^\bullet$ a finite subset. Associated to $S$ is the subgroup $L_S$ given by
\[ L_S = \ba_{\gamma \in S} \overline{\innp{\gamma}}. \]
We define the \emph{least common multiple of $S$} to be the set
\[ \LCM_{\Gamma,X}(S) = \set{\delta \in L_S^\bullet~:~ \norm{\delta}_X \leq \norm{\eta}_X \text{ for all }\eta \in L_S^\bullet}. \]
That is, $\LCM_{\Gamma,X}(S)$ is the set of nontrivial words in $L_S$ of minimal length in a fixed generating set $X$ of $\Gamma$. Finally, we set
\[
\lcm_{\Gamma,X}(S) =
\begin{cases}  \norm{\delta}_X& \text{ if there exists }\delta \in \LCM_{\Gamma,X}(S), \\
0 & \text{ if }\LCM_{\Gamma,X}(S) = \emptyset.
\end{cases}
\]

\nid The following basic lemma shows the importance of least common multiples in the study of both $\DM_\Gamma$ and $\G_\Gamma$.

\begin{lemma}\label{WordLengthForLCM}
Let $S \su \Gamma^\bullet$ be a finite set and $\delta \in \Gamma^\bullet$ have the following property: For any homomorphism $\vp\co \Gamma \to Q$, if $\ker \vp \cap S \ne \es$, then $\delta \in \ker \vp$. Then $\lcm_{\Gamma,X}(S) \leq \norm{\delta}_X$.
\end{lemma}

\begin{proof}
To prove this, for each $\gamma \in S$, note that $\vp_\gamma\co \Gamma \to \Gamma/\overline{\innp{\gamma}}$ is homomorphism for which $\ker \vp_\gamma \cap S \ne \es$. By assumption, $\delta \in \ker \vp_\gamma$ and thus in $\overline{\innp{\gamma}}$ for each $\gamma \in S$. Therefore, $\delta \in L_S$ and the claim now follows from the definition of $\lcm_{\Gamma,X}(S)$.
\end{proof}

\section{Lower bounds for free groups}\label{FreeGroupGrowth}

\nid In this section, using least common multiples, we will prove Theorem \ref{basiclowerbound}.

\paragraph{\textbf{1. Construct short least common multiples.}}

We begin with the following proposition.

\begin{prop}\label{FreeCandidateLemma}
Let $\gamma_1,\dots,\gamma_n \in F_m^\bullet$ and $\norm{\gamma_j}_X \leq d$ for all $j$. Then
\[ \lcm_{F_m,X}(\gamma_1,\dots,\gamma_n) \leq 6dn^2. \]
\end{prop}

\nid In the proof below, the reader will see that the important fact that we utilize is the following. For a pair of non-trivial elements $\gamma_1,\gamma_2$ in a nonabelian free group, we can conjugate $\gamma_1$ by a generator $\mu \in X$ to ensure that $\mu^{-1}\gamma_1\mu$ and $\gamma_2$ do not commute. This fact will be used repeatedly.

\begin{proof}
Let $k$ be the smallest natural number such that $n \leq 2^k$ (the inequality $2^k \leq 2n$ also holds). We will construct an element $\gamma$ in $L_{\set{\gamma_1,\dots,\gamma_n}}$ such that
\[ \norm{\gamma}_X \leq 6d4^k. \]
By Lemma \ref{WordLengthForLCM}, this implies the inequality asserted in the statement of the proposition. To this end, we augment the set $\set{\gamma_1,\dots,\gamma_n}$ by adding enough additional elements $\mu \in X$ such that our new set has precisely $2^k$ elements that we label $\set{\gamma_1,\dots,\gamma_{2^k}}$. Note that it does not matter if the elements we add to the set are distinct. For each pair $\gamma_{2i-1},\gamma_{2i}$, we replace $\gamma_{2i}$ by a conjugate $\mu_i\gamma_{2i}\mu_i^{-1}$ for $\mu_i \in X$ such that $[\gamma_{2i-1},\mu_i^{-1}\gamma_{2i}\mu_i]\ne 1$ and in an abuse of notation, continue to denote this by $\gamma_{2i}$. We define a new set of elements $\set{\gamma_i^{(1)}}$ by setting $\gamma_i^{(1)} = [\gamma_{2i-1},\gamma_{2i}]$. Note that $\norm{\gamma_i^{(1)}}_X \leq 4(d+2)$. We have $2^{k-1}$ elements in this new set and we repeat the above, again replacing $\gamma_{2i}^{(1)}$ with a conjugate by $\mu_i^{(1)}\in X$ if necessary to ensure that $\gamma_{2i-1}^{(1)}$ and $\gamma_{2i}^{(1)}$ do not commute. This yields $2^{k-2}$ non-trivial elements $\gamma_i^{(2)}=[\gamma_{2i-1}^{(1)},\gamma_{2i}^{(1)}]$ with $\norm{\gamma_i^{(2)}}_X \leq 4(4(d+2)+2)$. Continuing this inductively, at the $k$--stage we obtain an element $\gamma_1^{(k)} \in L_S$ such that
\[ \norm{\gamma_1^{(k)}}_X \leq 4^kd + a_k, \]
where $a_k$ is defined inductively by $a_0=0$ and
\[ a_j = 4(a_{j-1}+2). \]
The assertion
\[ a_j = 2\pr{\sum_{\ell=1}^j 4^\ell}, \]
is validated with an inductive proof. Thus, we have
\[ \norm{\gamma_1^{(k)}}_X \leq 4^kd+a_k \leq 3\pr{4^kd + 4^k} \leq 6d(4^k). \]
\end{proof}

\nid An immediate corollary of Proposition \ref{FreeCandidateLemma} is the following.

\begin{cor}\label{PrimeNumberTheorem}
\[ \lcm_{F_m,X}(B_{F^m,X}^\bullet(n)) \leq 6n(\word_{F_m,X}(n))^2. \]
\end{cor}

\paragraph{\textbf{2. Proof of Theorem \ref{basiclowerbound}.}}

We now give a short proof of Theorem \ref{basiclowerbound}. We begin with the following proposition.

\begin{prop}\label{freelowerbound}
Let $\Gamma$ be a nonabelian free group of rank $m$. Then $n^{1/3} \preceq \DM_{\Gamma,X}(n)$. 
\end{prop}

\begin{proof}
For $x \in X$, set
\[ S = \set{x,x^2,\dots,x^{n}}. \]
By Proposition \ref{FreeCandidateLemma}, if $\delta \in \LCM_{F_m,X}(S)$, then
\[ \norm{\delta}_X \leq 6n^3. \]
On the other hand, if $\vp\co F_m \to Q$ is a surjective homomorphism with $\vp(\delta) \ne 1$, the restriction of $\vp$ to $S$ is injective. In particular,
\[ \D_{F_m,X}^\lhd(\delta) \geq n. \]
In total, this shows that $n^{1/3} \preceq \DM_{F_m,X}$.
\end{proof}

\nid We now prove Theorem \ref{basiclowerbound}.

\begin{proof}[Proof of Theorem \ref{basiclowerbound}]
Let $\Gamma$ be a finitely generated group with finite generating set $X$. By assumption, $\Gamma$ contains a nonabelian free group $\Delta$. By passing to a subgroup, we may assume that $\Delta$ is finitely generated with free generating set $Y$. According to Lemma \ref{DivisibilityAsymptoticLemma} (b), we know that $\DM_{\Delta,Y}(n) \preceq \DM_{\Gamma,X}(n)$. By Proposition \ref{freelowerbound}, we also have $n^{1/3} \preceq \DM_{\Delta,Y}(n)$. The marriage of these two facts yields Theorem \ref{basiclowerbound}.
\end{proof}

\paragraph{\textbf{3. The basic girth inequality.}}

\nid We are now ready to prove (\ref{BasicGirthEquation}) for free groups. Again, for the reader's convenience, recall that (\ref{BasicGirthEquation}) is
\[ \G_{F_m,X}(n/2) \leq \DM_{F_m,X}(n/2)\pr{6n (\word_{F_m,X}(n))^{2}}. \]

\begin{proof}[Proof of (\ref{BasicGirthEquation})]
Let $\delta \in \LCM(B_{F_m,X}^\bullet(n))$ and let $Q$ be a finite group of order $\D_{F_m,X}^\lhd(\delta)$ such that there exists a homomorphism $\vp\co F_m \to Q$ with $\vp(\delta)\ne 1$. Since $\delta \in L_{B_{F_m,X}(n)}$, for each $\gamma$ in $B_{F_m,X}^\bullet(n)$, we also know that $\vp(\gamma) \ne 1$. In particular, it must be that $\vp$ restricted to $B_{F_m,X}^\bullet(n/2)$ is injective. The definitions of $\G_{F_m,X}$ and $\DM_{F_m,X}$ with Corollary \ref{PrimeNumberTheorem} yields
\[ \G_{F_m,X}(n/2) \leq \D^\lhd_{F_m,X}(\delta) \leq \DM_{F_m,X}(\norm{\delta}_X)\leq  \DM_{F_m,X}(6n(\word_{F_m,X}(n))^2), \]
and thus the desired inequality.
\end{proof}

\paragraph{\textbf{4. Proof of Theorem \ref{GirthPolynomialGrowth}.}}

We are also ready to prove Theorem \ref{GirthPolynomialGrowth}.

\begin{proof}
We must show that a finitely generated group $\Gamma$ is virtually nilpotent if and only if $\G_{\Gamma,X}$ has at most polynomial growth.
If $\G_{\Gamma,X}$ is bounded above by a polynomial in $n$, as $\word_{\Gamma,X} \leq \G_{\Gamma,X}$, it must be that $\word_{\Gamma,X}$ is bounded above by a polynomial in $n$. Hence, by Gromov's Polynomial Growth Theorem, $G$ is virtually nilpotent.\smallskip\smallskip

\nid Suppose now that $\Gamma$ is virtually nilpotent and set $\Gamma_{\textrm{Fitt}}$ to be the Fitting subgroup of $\Gamma$. It is well known (see \cite{Dek}) that $\Gamma_{\textrm{Fitt}}$ is torsion free and finite index in $\Gamma$. By Lemma \ref{GirthAsymptoticLemma} (c), we may assume that $\Gamma$ is torsion free. In this case, $\Gamma$ admits a faithful, linear representation $\psi$ into $\B{U}(d,\Z)$, the group of upper triangular, unipotent matrices with integer coefficients in $\GL(d,\Z)$ (see \cite{Dek}). Under this injective homomorphism, the elements in $B_{\Gamma,X}(n)$ have matrix entries with norm bounded above by $Cn^k$, where $C$ and $k$ only depends on $\Gamma$. Specifically, we have
\[ \abs{(\psi(\gamma))_{i,j}} \leq C\norm{\gamma}_X^k. \]
This is a consequence of the Hausdorff--Baker--Campbell formula (see \cite{Dek}).
Let $r$ be the reduction homomorphism
\[ r \co \B{U}(d,\Z) \lra \B{U}(d,\Z/ 2Cn^k \Z) \]
defined by reducing matrix coefficients modulo $2 Cn^k$. By selection, the restriction of $r$ to $B_{\Gamma,X}^\bullet(n)$ is injective. So we have
\begin{equation}\label{CardinalityInequality}
\abs{r(\psi(\Gamma))} \leq  \abs{\B{U}(d,\Z/ 2Cn^k \Z)} \leq (2Cn^k)^{d^2}.
\end{equation}
This inequality gives
\[ \G_{\Gamma,X}(n) \leq (2Cn^k)^{d^2} = C_1n^{kd^2}. \]
Therefore, $\G_{\Gamma,X}(n)$ is bounded above by a polynomial function in $n$ as claimed.
\end{proof}

\paragraph{\textbf{5. Generalities.}}

The results and methods for the free group in this section can be generalized. Specifically, we require the following two properties:
\begin{itemize}
\item[(i)]
$\Gamma$ has an element of infinite order.
\item[(ii)]
For all non-trivial $\gamma_1,\gamma_2 \in \Gamma$, there exists $\mu_{1,2} \in X$ such that $[\gamma_1,\mu_{1,2}\gamma_2\mu_{1,2}^{-1}]\ne 1$.
\end{itemize}

\nid
With this, we can state a general result established with an identical method taken for the free group.

\begin{thm}
Let $\Gamma$ be finitely generated group that satisfies (i) and (ii). Then
\begin{itemize}
\item[(a)]
$\G_{\Gamma,X}(n/2) \leq \DM_{\Gamma,X}(n/2)\pr{6n (\word_{\Gamma,X}(n))^{2}}$.
\item[(b)]
$n^{1/3} \preceq \DM_{\Gamma,X}$.
\end{itemize}
\end{thm}

\section{The proof of Theorem \ref{toughlowerbound}}
\label{toughlowerboundSection}

\nid In this section we prove Theorem \ref{toughlowerbound}. For sake of clarity, before commencing with the proof, we outline the basic strategy. We will proceed via contradiction, assuming that $\max \D_{F_m}(n) \preceq \log n$. We will apply this assumption to a family of test elements $\delta_n$ derived from least common multiples of certain simple sets $S(n)$ to produce a family of finite index subgroups $\Delta_n$ in $F_m$. Employing the Prime Number Theorem, we will obtain upper bounds (see (\ref{LinearBound}) below) for the indices $[F_m:\Delta_n]$. Using covering space theory and a simple albeit involved inductive argument, we will derive the needed contradiction by showing the impossibility of these bounds. The remainder of this section is devoted to the details.

\begin{proof}[Proof of Theorem \ref{toughlowerbound}]
Our goal is to show $\max \D_{F_m}(n) \npreceq \log(n)$ for $m \geq 2$. By Lemma 1.1 in \cite{Bou}, it suffices to show this for $m=2$. To that end, set $\Gamma = F_2$ with free generating set $X=\set{x,y}$, and
\[ S(n) = \set{x,x^2,\dots,x^{\lcm(1,\dots,n)}}. \]
We proceed by contradiction, assuming that $\max \D_\Gamma(n) \preceq \log(n)$. By definition, there exists a constant $C>0$ such that $\max \D_\Gamma(n) \leq C\log(Cn)$ for all $n$. For any $\delta_n \in \LCM_{\Gamma,X}(S(n))$, this implies that there exists a finite index subgroup $\Delta_n < \Gamma$ such that $\delta_n \notin \Delta_n$ and
\[ [\Gamma:\Delta_n] \leq C\log(C\norm{\delta_n}_X). \]
According to Proposition \ref{FreeCandidateLemma}, we also know that
\[ \norm{\delta_n}_X \leq D(\lcm(1,\dots,n))^3. \]
In tandem, this yields
\[ [\Gamma:\Delta_n] \leq C\log(CD(\lcm(1,\dots,n))^3). \]
By the Prime Number Theorem, we have
\[ \lim_{n \to \iny} \frac{\log(\lcm(1,\dots,n))}{n} = 1. \]
Therefore, there exists $N>0$ such that for all $n \geq N$
\[ \frac{n}{2} \leq \log(\lcm(1,\dots,n)) \leq \frac{3n}{2}. \]
Combining this with the above, we see that there exists a constant $M>0$ such that for all $n\geq N$,
\begin{equation}\label{LinearBound}
[\Gamma:\Delta_n] \leq C\log(CD) + \frac{9Cn}{2} \leq Mn.
\end{equation}
Our task now is to show (\ref{LinearBound}) cannot hold. In order to achieve the desired contradiction, we use covering space theory. With that goal in mind, let $S^1 \vee S^1$ be the wedge product of two circles and recall that we can realize $\Gamma$ as $\pi_1(S^1 \vee S^1,*)$ by identifying $x,y$ with generators for the fundamental groups of the respective pair of circles. Here, $*$ serves as both the base point and the identifying point for the wedge product. According to covering space theory, associated to the conjugacy class $[\Delta_n]$ of $\Delta_n$ in $\Gamma$, is a finite cover $Z_n$ of $S^1 \vee S^1$ of covering degree $[\Gamma:\Delta_n]$ (unique up to covering isomorphisms). Associated to a conjugacy class $[\gamma]$ in $\Gamma$ is a closed curve $c_\gamma$ on $S^1 \vee S^1$. The distinct lifts of $c_\gamma$ to $Z_n$ correspond to the distinct $\Delta_n$--conjugacy classes of $\gamma$ in $\Gamma$. The condition that $\gamma \notin \Delta_n$ implies that at least one such lift cannot be a closed loop.\smallskip\smallskip

\nid Removing the edges of $Z_n$ associated to the lifts of the closed curve associated to $[y]$, we get a disjoint union of topological circles, each of which is a union of edges associated to the lifts of the loop associated to $[x]$. We call these circles $x$--cycles and say the length of an $x$--cycle is the total number of edges of the cycle. The sum of the lengths over all the distinct $x$--cycles is precisely $[\Gamma:\Delta_n]$. For an element of the form $x^\ell$, each lift of the associated curve $c_{x^\ell}$ is contained on an $x$--cycle. Using elements of the form $x^\ell$, we will produce enough sufficiently long $x$--cycles in order to contradict (\ref{LinearBound}).\smallskip\smallskip

\nid We begin with the element $x^{\lcm(1,\dots,m)}$ for $1 \leq m \leq n$. This will serve as both the base case for an inductive proof and will allow us to introduce some needed notation. By construction, some $\Gamma$--conjugate of $x^{\lcm(1,\dots,m)}$ is not contained in $\Delta_n$. Indeed, $x^\ell$ for any $1 \leq \ell \leq \lcm(1,\dots,n)$ is never contained in the intersection of all conjugates of $\Delta_n$. Setting $c_m$ to be the curve associated to $x^{\lcm(1,\dots,m)}$, this implies that there exists a lift of $c_m$ that is not closed in $Z_n$. Setting $C_n^{(1)}$ to be the $x$--cycle containing this lift, we see that the length of $C_n^{(1)}$ must be at least $m$. Otherwise, some power $x^\ell$ for $1 \leq \ell \leq m$ would have a closed lift for this base point and this would force this lift of $c_m$ to be closed. Setting $k_{n,m}^{(1)}$ to be the associated length, we see that $m \leq k_{n,m}^{(1)} \leq Mn$ when $n \geq N$.\smallskip\smallskip

\nid Using the above as the base case, we claim the following:\smallskip\smallskip

\nid \textbf{Claim.} \emph{For each positive integer $i$, there exists a positive integer $N_i \geq N$ such that for all $n \geq 8N_i$, there exists disjoint $x$--cycles $C_n^{(1)},\dots,C_n^{(i)}$ in $Z_n$ with respective lengths $k_n^{(1)},\dots,k_n^{(i)}$ such that $k_n^{(j)} \geq n/8$ for all $1 \leq j \leq i$.}\smallskip\smallskip

\nid That this claim implies the desired contradiction is clear. Indeed, if the claim holds, we have
\[ \frac{ni}{8} \leq \sum_{j=1}^i k_n^{(j)} \leq [\Gamma:\Delta_n] \]
for all positive integers $i$ and all $n \geq N_i$. Taking $i > 8M$ yields an immediate contradiction of (\ref{LinearBound}). Thus, we are reduced to proving the claim.

\begin{proof}[Proof of Claim]
For the base case $i=1$, we can take $N_1=N$ and $m=n$ in the above argument and thus produce an $x$--cycle of length $k_n^{(1)}$ with $n \leq k_n^{(1)}$ for any $n \geq N_1$. Proceeding by induction on $i$, we assuming the claim holds for $i$. Specifically, there exists $N_i \geq N$ such that for all $n \geq 8N_i$, there exists disjoint $x$--cycles $C_n^{(1)},\dots,C_n^{(i)}$ in $Z_n$ with lengths $k_n^{(j)} \geq n/8$. By increasing $N_i$ to some $N_{i+1}$, we need to produce a new $x$--cycle $C_n^{(i+1)}$ in $Z_n$ of length $k_n^{(i+1)} \geq n/8$ for all $n \geq 8N_{i+1}$. For this, set
\[ \ell_{n,m} = \lcm(1,\dots,m)\prod_{j=1}^i k_n^{(j)}. \]
By construction, the lift of the closed curve associated to $x^{\ell_{n,m}}$ to each cycle $C_n^{(j)}$ is closed. Consequently, any lift of the curve associated to $x^{\ell_{n,m}}$ that is not closed must necessarily reside on an $x$--cycle that is disjoint from the previous $i$ cycles $C_n^{(1)},\dots, C_n^{(i)}$. In addition, we must ensure that this new $x$--cycle has length at least $n/8$. To guarantee that the curve associated to $x^{\ell_{n,m}}$ has a lift that is not closed, it is sufficient to have the inequality
\begin{equation}\label{NonClosedLift}
\ell_{n,m} \leq \lcm(1,\dots,n).
\end{equation}
In addition, if $m\geq n/8$, then the length of $x$--cycle containing this lift must be at least $n/8$. We focus first on arranging (\ref{NonClosedLift}). For this, since $k_n^{(j)} \leq Mn$ for all $j$, (\ref{NonClosedLift}) holds if
\[ (Mn)^i\lcm(1,\dots,m) \leq \lcm(1,\dots,n). \]
This, in turn, is equivalent to
\[ \log(\lcm(1,\dots,m)) \leq \log(\lcm(1,\dots,n)) - i\log(Mn). \]
Set $N_{i+1}$ to be the smallest positive integer such that
\[ \frac{n}{8} - i\log(Mn) > 0 \]
for all $n \geq 8N_{i+1}$. Taking $n>8N_{i+1}$ and $n/8 \leq m \leq n/4$, we see that
\begin{align*}
\log(\lcm(1,\dots,m)) &\leq \frac{3m}{2} \\
&\leq \frac{3n}{8} \\
&\leq \frac{3n}{8} + \pr{\frac{n}{8}-i\log(Mn)} \\
&= \frac{n}{2} - i\log(Mn) \\
&\leq \log(\lcm(1,\dots,n)) - i\log(Mn).
\end{align*}
In particular, we produce a new $x$--cycle $C_n^{(i+1)}$ of length $k_n^{(i+1)}\geq n/8$ for all $n \geq N_{i+1}$.
\end{proof}

\nid Having proven the claim, our proof of Theorem \ref{toughlowerbound} is complete.
\end{proof}

\nid Just as in Theorem \ref{basiclowerbound}, Theorem \ref{toughlowerbound} can be extended to any finitely generated group that contains a nonabelian free subgroup.

\begin{cor}
Let $\Gamma$ be a finitely generated group that contains a nonabelian free subgroup. Then
\[ \max \D_{\Gamma,X}(n) \npreceq \log(n). \]
\end{cor}


\noindent Department of Mathematics \\
University of Chicago \\
Chicago, IL 60637, USA \\
email: {\tt khalid@math.uchicago.edu}, {\tt dmcreyn@math.uchicago.edu}\\


\end{document}